\documentclass[reqno]{amsart}
\usepackage{amssymb}

\newtheorem{theorem}{Theorem}

\newtheorem{lemma}[theorem]{Lemma}
\newtheorem{corollary}[theorem]{Corollary}

\theoremstyle{remark}
\newtheorem{remark}{Remark}

\numberwithin{equation}{section}

\begin{document}

\title[Branching Formula for Macdonald-Koornwinder Polynomials]
{Branching Formula for Macdonald-Koornwinder Polynomials}

\author{J.F.  van Diejen}

\address{
Instituto de Matem\'atica y F\'{\i}sica, Universidad de Talca,
Casilla 747, Talca, Chile}

\email{diejen@inst-mat.utalca.cl}

\author{E. Emsiz}

\address{
Facultad de Matem\'aticas, Pontificia Universidad Cat\'olica de Chile,
Casilla 306, Correo 22, Santiago, Chile}
\email{eemsiz@mat.puc.cl}


\thanks{This work was supported in part by the {\em Fondo Nacional de Desarrollo
Cient\'{\i}fico y Tecnol\'ogico (FONDECYT)} Grants \# 1130226 and  \# 1141114.}

\date{August 2014}

\begin{abstract}
We present an explicit branching formula for the six-parameter Macdonald-Koornwinder polynomials with hyperoctahedral symmetry.
\end{abstract}

\maketitle

\section{Introduction}\label{sec1}
Branching formulas constitute a powerful tool in algebraic combinatorics providing a recursive scheme to build symmetric polynomials via induction in the number of variables
\cite{mac:symmetric,las-war:branching}. 
The combinatorial aspects of the hyperoctahedral-symmetric Macdonald-Koornwinder polynomials \cite{mac:orthogonal,koo:askey-wilson}  were explored in seminal  works of Okounkov and Rains
\cite{ouk:bc-type,rai:bcn-symmetric}.   In particular, the structure of a branching formula for the Macdonald-Koornwinder polynomials has been outlined at the end of \cite[Sec.~5]{rai:bcn-symmetric}. The aim of the present note is to make the branching polynomials under consideration explicit. Following the ideas underlying the proof of the branching formula for the Macdonald polynomials \cite[Ch.~VI.7]{mac:symmetric},  our main tools to achieve this goal consist of:
Mimachi's Cauchy formula for the Macdonald-Koornwinder polynomials \cite{mim:duality}, (a special `column-row' case of) the Cauchy formula for Okounkov's  hyperoctahedral interpolation polynomials \cite{ouk:bc-type}, and explicitly known Pieri coefficients for the Macdonald-Koornwinder polynomials \cite{die:self-dual,die:properties}.

The material is structured as follows. After recalling some necessary preliminaries regarding the Macdonald-Koorwinder polynomials and their Pieri formulas in Section \ref{sec2},  our branching formula is first stated in Section \ref{sec3} and then proven in Section \ref{sec4}.

\section{Preliminaries}\label{sec2} 
\subsection{Macdonald-Koornwinder polynomials  \cite{koo:askey-wilson}}
For a partition
\begin{equation}
\lambda\in\Lambda_n:=\{(\lambda_1,\ldots, \lambda_n)\in\mathbb{Z}^n\mid \lambda_1\geq\lambda_2\geq \cdots \geq\lambda_n\geq 0\} ,
\end{equation}
the monic Macdonald-Koornwinder polynomial
$P_\lambda(z_1,\ldots ,z_n;q,t,\mathbf{t})$ is a
Laurent polynomial in the complex variables $z_1,\ldots ,z_n$  that depends rationally on the parameters $q,t$ and $\mathbf{t}:=(t_0,t_1,t_2,t_3)$. It is determined by a leading monomial of the form $z_1^{\lambda_1}z_2^{\lambda_2}\cdots z_n^{\lambda_n}$, while being symmetric with respect to the action of the hyperoctahedral group $W=S_n\ltimes \mathbb{Z}_2^n$ by permutations and inversions of the variables. For
real parameter values in the domain 
 $0<q,|t|,|t_l|<1$ ($l=0,1,2,3$), the Macdonald-Koornwinder polynomials form an orthogonal system on the $n$-dimensional torus $|z_j|=1$, $j=1,\ldots ,n$. The orthogonality measure is given by Gustafson's $q$-Selberg type density \cite{gus:generalization}
$$
\Delta =
\prod_{1\leq j\leq n} \frac{(z_j^2,z_j^{-2};q)_\infty }{\prod_{0\leq l\leq 3} (t_l z_j,t_l z_j^{-1};q)_\infty}
\prod_{1\leq j<k\leq n} \frac{(z_jz_k,z_jz_k^{-1},z_j^{-1}z_k,z_j^{-1}z_k^{-1};q)_\infty }{(tz_jz_k,tz_jz_k^{-1},tz_j^{-1}z_k,tz_j^{-1}z_k^{-1};q)_\infty }
$$
with respect to the Haar measure on this torus.
Here and below we employ standard conventions for the $q$-Pochhammer symbols: $(a;q)_k:=(1-a)(1-aq)\cdots (1-aq^{k-1})$ (with $(a;q)_0:= 1$) and $(a_1,\ldots,a_l;q)_k:=(a_1;q)_k\cdots (a_l;q)_k$.

\subsection{Pieri coefficients \cite{die:self-dual,die:properties,sah:nonsymmetric}}
The $W$-invariant Laurent polynomials
\begin{equation}
E_r (z_1,\ldots ,z_n;t,t_0)=\sum_{1\leq j_1<\cdots < j_r\leq n} \langle   z_{j_1} ; t^{j_1-1}t_0\rangle \cdots \langle   z_{j_r} ; t^{j_r-1}t_0\rangle 
\end{equation}
$ (r=1,\ldots ,n )$, with $\langle z;x\rangle:=z+z^{-1}-x-x^{-1}$, are special instances of Okounkov's hyperoctahedral interpolation polynomials \cite{ouk:bc-type}---with shifted variables as considered by
Rains \cite{rai:bcn-symmetric}---that correspond to the partitions with only a single column  \cite{kom-nou-shi:kernel}. They describe the eigenvalues of commuting difference operators diagonalized by the Macdonald-Koornwinder polynomials
\cite{die:commuting}  and are also instrumental in Ito's Aomoto-style proof \cite{ito:q-difference} of the hyperoctahedral
${}_6\Psi_6$ sum evaluated in Ref. \cite{die:certain}.

Let $C^{\mu,n}_{\lambda ,r }(q,t,\mathbf{t})$ denote the coefficients in the Macdonald-Koornwinder Pieri-expansions associated with these one-column interpolation polynomials:
\begin{equation}\label{pieri}
E_r (z_1,\ldots ,z_n;t,t_0) P_\lambda(z_1,\ldots ,z_n;q,t,\mathbf{t})  = \sum_{\substack{\mu\in\Lambda_n \\ \mu\sim_r\lambda}}C^{\mu,n}_{\lambda , r}(q,t,\mathbf{t})P_\mu (z_1,\ldots ,z_n;q,t,\mathbf{t}) , 
\end{equation}
$r=1,\ldots ,n$. To filter only the nonvanishing coefficients, we have employed the following proximity relation within $\Lambda_n$  restricting the sum on the RHS:
$\mu\sim_r\lambda$ 
iff there exists a partition $\nu\in\Lambda_n$ with $\nu\subset \lambda$ and $\nu\subset \mu$ such that the skew diagrams $ \lambda /\nu $ and $ \mu / \nu$ are vertical strips with
 $| \lambda /\nu| +| \mu/\nu |\leq r$. Here $|\cdot |$ refers to the number of boxes of the diagram, $\nu\subset\lambda$ means that $\nu\in\Lambda_n$ is contained in $\lambda$: $\nu_j\leq \lambda_j$ ($j=1,\ldots ,n$), and (recall) the skew diagram $\lambda/\nu$ is a vertical strip iff $\nu_j\leq\lambda_j\leq\nu_j+1$ ($j=1,\ldots ,n$). 
 
Upon writing $J=\{ 1\leq j\leq n \mid \lambda_j\neq \mu_j\}$, $J^c=\{1,\ldots ,n\}\setminus J$,  and
$\epsilon_j=\mu_j-\lambda_j$ for $j\in J$ (so, if $\mu\sim_r\lambda$
the cardinality $|J|$ of $J$ is at most $ r$ and $\epsilon_j\in \{ 1,-1\}$),
one can express the Pieri coefficients in question explicitly as follows:
\begin{equation}
C^{\mu,n}_{\lambda , r}(q,t,\mathbf{t})= \frac{P_\lambda(\hat{\tau}_1,\ldots ,\hat{\tau}_n;q,t,\mathbf{t})}{P_\mu(\hat{\tau}_1,\ldots ,\hat{\tau}_n;q,t,\mathbf{t})} V^{n}_{\epsilon J} (\lambda ;q,t,\mathbf{t}) U^{n}_{J^c,r-|J|} (\lambda ; q,t,\mathbf{t}) ,
\end{equation}
where
\begin{eqnarray*}
\lefteqn{P_\lambda(\hat{\tau}_1,\ldots ,\hat{\tau}_n;q,t,\mathbf{t})=}&& \\
&& \prod_{1\leq j\leq n}  \frac{\prod_{0\leq l\leq 3} (\hat{t}_l\hat{\tau}_j ;q)_{\lambda_j }}{\tau_j^{\lambda_j} (\hat{\tau}_j^2 ;q)_{2\lambda_j}} 
\prod_{1\leq j<k\leq n} \frac{(t\hat{\tau}_j\hat{\tau}_k;q)_{\lambda_j+\lambda_k} (t\hat{\tau}_j\hat{\tau}_k^{-1};q)_{\lambda_j-\lambda_k}}{(\hat{\tau}_j\hat{\tau}_k;q)_{\lambda_j+\lambda_k}(\hat{\tau}_j\hat{\tau}_k^{-1};q)_{\lambda_j-\lambda_k}}  ,
\end{eqnarray*}

\begin{align*}
V^{n}_{\epsilon J} (\lambda ;q,t,\mathbf{t})&=\prod_{j\in J} \frac{\prod_{0\leq l\leq 3}   (1-\hat{t}_l\hat{\tau}_j^{\epsilon_j}q^{\epsilon_j\lambda_j})}{t_0(1-\hat{\tau}_j^{2\epsilon_j}q^{2\epsilon_j\lambda_j})(1-\hat{\tau}_j^{2\epsilon_j}q^{2\epsilon_j\lambda_j+1})} \\
&\times \prod_{\substack{j,j^\prime\in J\\ j<j^\prime}}
\frac{(1-t\hat{\tau}_j^{\epsilon_j}\hat{\tau}_{j^\prime}^{\epsilon_{j^\prime}}q^{\epsilon_j\lambda_j+\epsilon_{j^\prime}\lambda_{j^\prime}})(1-t\hat{\tau}_j^{\epsilon_j}\hat{\tau}_{j^\prime}^{\epsilon_{j^\prime}}q^{\epsilon_j\lambda_j+\epsilon_{j^\prime}\lambda_{j^\prime}+1})}{t(1-\hat{\tau}_j^{\epsilon_j}\hat{\tau}_{j^\prime}^{\epsilon_{j^\prime}}q^{\epsilon_j\lambda_j+\epsilon_{j^\prime}\lambda_{j^\prime}})(1-\hat{\tau}_j^{\epsilon_j}\hat{\tau}_{j^\prime}^{\epsilon_{j^\prime}}q^{\epsilon_j\lambda_j+\epsilon_{j^\prime}\lambda_{j^\prime}+1})} \\
&\times \prod_{j\in J,k\in J^c} \frac{(1-t\hat{\tau}_j^{\epsilon_j}\hat{\tau}_k q^{\epsilon_j\lambda_j+\lambda_k})(1-t\hat{\tau}_j^{\epsilon_j}\hat{\tau}_k^{-1}q^{\epsilon_j\lambda_j-\lambda_k})}{t(1-\hat{\tau}_j^{\epsilon_j}\hat{\tau}_k q^{\epsilon_j\lambda_j+\lambda_k})(1-\hat{\tau}_j^{\epsilon_j}\hat{\tau}_k^{-1}q^{\epsilon_j\lambda_j-\lambda_k})}
 ,
\end{align*}
and
\begin{align*}
 U^{n}_{K,p} (\lambda ;q,t,\mathbf{t})&=  (-1)^p \sum_{\substack{I\subset K,\, |I|=p \\ \epsilon_i\in \{ 1 , -1\} , i\in I}}
 \Biggl( \prod_{i\in I} \frac{\prod_{0\leq l\leq 3}   (1-\hat{t}_l\hat{\tau}_i^{\epsilon_i}q^{\epsilon_i\lambda_i})}{t_0(1-\hat{\tau}_i^{2\epsilon_i}q^{2\epsilon_i\lambda_i})(1-\hat{\tau}_i^{2\epsilon_i}q^{2\epsilon_i\lambda_i+1})} \\
&\times \prod_{\substack{i,i^\prime\in I\\ i<i^\prime}}
\frac{(1-t\hat{\tau}_i^{\epsilon_i}\hat{\tau}_{i^\prime}^{\epsilon_{i^\prime}}q^{\epsilon_i\lambda_i+\epsilon_{i^\prime}\lambda_{i^\prime}})(1-t^{-1}\hat{\tau}_i^{\epsilon_i}\hat{\tau}_{i^\prime}^{\epsilon_{i^\prime}}q^{\epsilon_i\lambda_i+\epsilon_{i^\prime}\lambda_{i^\prime}+1})}{(1-\hat{\tau}_i^{\epsilon_i}\hat{\tau}_{i^\prime}^{\epsilon_{i^\prime}}q^{\epsilon_i\lambda_i+\epsilon_{i^\prime}\lambda_{i^\prime}})(1-\hat{\tau}_i^{\epsilon_i}\hat{\tau}_{i^\prime}^{\epsilon_{i^\prime}}q^{\epsilon_i\lambda_i+\epsilon_{i^\prime}\lambda_{i^\prime}+1})} \\
&\times \prod_{i\in I,k\in K\setminus I} \frac{(1-t\hat{\tau}_i^{\epsilon_i}\hat{\tau}_k q^{\epsilon_i\lambda_i+\lambda_k})(1-t\hat{\tau}_i^{\epsilon_i}\hat{\tau}_k^{-1}q^{\epsilon_i\lambda_i-\lambda_k})}{t(1-\hat{\tau}_i^{\epsilon_i}\hat{\tau}_k q^{\epsilon_i\lambda_i+\lambda_k})(1-\hat{\tau}_i^{\epsilon_i}\hat{\tau}_k^{-1}q^{\epsilon_i\lambda_i-\lambda_k})} \Biggr)
\end{align*}
for $p=0,\ldots ,|K|$ (with the convention that $V^{n}_{\epsilon J} (\lambda ;q,t,\mathbf{t})=1$ if $J$ is empty and 
 $U^{n}_{K,p} (\lambda ;q,t,\mathbf{t})=1$ if $p=0$).
Here $$\tau_j=t^{n-j}t_0, \quad \hat{\tau}_j=t^{n-j}\hat{t}_0 \quad (j=1,\ldots ,n),$$ and
\begin{equation*}
\hat{t}_0^2=q^{-1}t_0t_1t_2t_3,\qquad \hat{t}_0\hat{t}_l=t_0t_l\quad (l=1,2,3) .
\end{equation*}

\begin{remark}
Below we will employ a trivially extended notion of
$E_r (z_1,\ldots ,z_n;t,t_0)$ and $C^{\mu,n}_{\lambda , r}(q,t,\mathbf{t})$ that allows $r$ and $n$ to become equal to zero. By convention $E_0 (z_1,\ldots ,z_n;t,t_0):=1$,  whence for the corresponding coefficients $C^{\mu,n}_{\lambda , 0}(q,t,\mathbf{t})=1$ if $\mu=\lambda$ and vanishes otherwise.
\end{remark}

\section{Branching formula}\label{sec3}
Let us recall that for $\mu\subset \lambda\in\Lambda_n$, the skew diagram $\lambda/\mu$ is a horizontal strip provided the parts of $\lambda$ and $\mu$ interlace as follows:
$$
\lambda_1\geq\mu_1\geq\lambda_2\geq\mu_2\geq\cdots \geq\lambda_n\geq\mu_n .
$$
We will need the following relation in $\Lambda_n$ expressing that the partition $\lambda$ can be obtained form $\mu\subset\lambda$ by adding at most two horizontal strips:  $\mu \preceq \lambda$ iff there exists a $\nu\in\Lambda_n$ with $\mu\subset\nu\subset\lambda$ such that the skew diagrams $\lambda/\nu$ and $\nu/\mu$ are horizontal strips. From now on we will think of $\Lambda_n$ as being embedded in $\Lambda_{n+1}$ in the natural way (i.e. `by adding a part of size zero'). The main result of this note is given by the following branching formula for the Macdonald-Koornwinder polynomials, the proof of which is delayed until Section \ref{sec4} below.

We denote by $m^n\in\Lambda_n$ the rectangular partition such that $(m^n)_j=m$ ($j=1,\ldots,n$)
and---more generally---by $m^n-\mu$ with $\mu\subset m^n$ the partition such that $(m^n-\mu)_j=m-\mu_{n+1-j}$ ($j=1,\ldots,n$). Finally, we write
$\lambda^\prime\in\Lambda_m$ ($m\geq  \lambda_1$) for the conjugate partition of $\lambda\in\Lambda_n$, i.e. with
$\lambda^\prime_i$ counting the number of parts of $\lambda$ that are greater or equal than $ i$ ($i=1,\ldots ,m$).
\begin{theorem}[Branching Formula]\label{branching:thm}
For $\lambda\in \Lambda_{n+1}$, the Macdonald-Koornwinder polynomial in $(n+1)$ variables expands in terms of the $n$-variable polynomials as
\begin{subequations}
\begin{equation}
P_\lambda (z_1,\ldots ,z_n,x; q ,t ,\mathbf{t})=
\sum_{\substack{\mu\in\Lambda_n\\ \mu \preceq \lambda}}  P_\mu (z_1,\ldots ,z_n ; q ,t ,\mathbf{t})
P_{\lambda/\mu}(x; q ,t ,\mathbf{t}),
\end{equation}
with one-variable branching polynomials of degree
$d=|\{ 1\leq j\leq m \mid \lambda_j^\prime=\mu_j^\prime +1\}|$ whose expansion
\begin{equation}\label{br-pol}
P_{\lambda/\mu}(x; q ,t ,\mathbf{t})=
\sum_{0\leq k\leq d} B_{\lambda /\mu}^k (q,t,\mathbf{t})
\langle x; t_0\rangle_{q,k}
\end{equation}
in the basis of the interpolation polynomials in one variable
\begin{equation}
\langle x; t_0\rangle_{q,k}:=\langle x; t_0\rangle \langle x; q t_0\rangle\cdots \langle x; q^{k-1}t_0\rangle \qquad (\text{with}\
\langle x; t_0\rangle_{q,0}:=1)
 \end{equation}
has coefficients that are
given explicitly by the Macdonald-Koornwinder Pieri coefficients in $m=\lambda_1$ variables:
\begin{equation}\label{br-coef}
B_{\lambda/\mu}^k (q,t,\mathbf{t})=
 (-1)^{k+|\lambda|-|\mu|} C^{(n+1)^m-\lambda^\prime,m}_{n^m-\mu^\prime,m-k}(t,q,\mathbf{t}) \quad (k=0,\ldots ,d).
 \end{equation}
 \end{subequations}
\end{theorem}

It is well-known \cite{koo:askey-wilson} that in the case of only a single variable the Macdonald-Koornwinder polynomials  reduce to the five-parameter  monic Askey-Wilson polynomials $P_m (z; q , \mathbf{t})$, $m=0,1,2,\ldots$ \cite{ask-wil:some}. 
On the other hand, if we formally set $n=0$ and $\mu=0$ then (the proof of) Theorem \ref{branching:thm} remains valid. This gives rise to the following expansion of the Askey-Wilson polynomials in terms of the one-variable interpolation polynomials $\langle x; t_0\rangle_{q,k}$.
\begin{corollary}[Askey-Wilson Polynomials]\label{aw:cor} The monic Askey-Wilson polynomial of degree $m$ is given by
\begin{subequations}
\begin{equation}\label{awp}
P_m (z; q , \mathbf{t}) =\sum_{0\leq k\leq m} B^k_{m/0}(q,\mathbf{t})\langle z; t_0\rangle_{q,k} 
\end{equation}
with
\begin{equation}\label{awc}
B^k_{m/0}(q,\mathbf{t})=(-1)^{m+k} C^{0^m,m}_{0^m,m-k}(t,q,\mathbf{t})  .
\end{equation}
\end{subequations}
\end{corollary}
This formula for the Askey-Wilson polynomials amounts to the $n=1$ case of Okounkov's binomial formula for the Macdonald-Koornwinder polynomials \cite[Thm. 7.1]{ouk:bc-type} with the binomial coefficients written explicitly in terms of the $m$-variable
Macdonald-Koornwinder Pieri coefficients. Notice that since the Askey-Wilson polynomials on the LHS are independent of $t$, it follows that the $t$-dependence of the corresponding branching coefficients drops out as well.
In this special situation, alternative expressions for
the relevant binomial coefficients are available in a much more compact form
\cite[Prp. 4.1]{rai:bcn-symmetric} and the binomial formula is in fact seen to reduce to the usual ${}_4\phi_3$ representation of the Askey-Wilson polynomial \cite[p. 25]{kom-nou-shi:kernel}.

By iterating the branching formula in Theorem \ref{branching:thm}, one finds the general Macdonald-Koornwinder branching polynomial as a sum of factorized contributions over ascending chains of partitions.
\begin{corollary}[Branching Polynomials]
For $\lambda\in \Lambda_{n+l}$, one has that
\begin{align}
&P_\lambda (z_1,\ldots ,z_n,x_1,\ldots, x_l ; q ,t ,\mathbf{t})= \\
&\sum_{\substack{\mu^{(i)}\in\Lambda_{n+i}, i=0,\ldots ,l\\ \mu=\mu^{(0)} \preceq \mu^{(1)}\preceq \cdots \preceq \mu^{(l)}=\lambda}}  P_\mu (z_1,\ldots ,z_n ; q ,t ,\mathbf{t})
\prod_{1\leq i\leq l} P_{\mu^{(i)}/\mu^{(i-1)}}(x_i; q ,t ,\mathbf{t}). \nonumber
\end{align}
\end{corollary}

Setting $n=1$ in the latter formula, leads us to an explicit formula for the Macdonald-Koorwinder polynomials generalizing the formula for the Askey-Wilson polynomials in Corollary \ref{aw:cor}.
\begin{corollary}[Macdonald-Koornwinder Polynomials]
For $\lambda\in \Lambda_{n}$, the monic Macdonald-Koornwinder polynomial is given by
\begin{equation}
P_\lambda (z_1,\ldots ,z_n; q ,t ,\mathbf{t})= 
\sum_{\substack{\mu^{(i)}\in\Lambda_{i}, i=1,\ldots ,n\\  \mu^{(1)}\preceq \mu^{(2)}\preceq \cdots \preceq \mu^{(n)}=\lambda}} \prod_{1\leq i\leq n} P_{\mu^{(i)}/\mu^{(i-1)}}(z_i; q ,t ,\mathbf{t}), 
\end{equation}
where $P_{\mu^{(1)}/\mu^{(0)}}(z; q ,t ,\mathbf{t}) := P_{\mu^{(1)}}(z; q ,\mathbf{t}) $ \eqref{awp}, \eqref{awc} by convention.
\end{corollary}
This is the analog of a classic formula for the usual permutation-symmetric Macdonald polynomials in terms of semistandard tableaux,
cf.  e.g. \cite[Ch. VI.7]{mac:symmetric}  and  \cite[Sec. 1]{las-war:branching}.

\begin{remark}\label{rem-m}
It follows from the proof in Section \ref{sec4} below that the branching formula in Theorem \ref{branching:thm} holds in fact {\em for any} $m\geq \lambda_1$, i.e. the expressions for the branching coefficients $B_{\lambda/\mu}^k$ in Eq. \eqref{br-coef}  {\em do not depend
on} $m\geq \lambda_1$.
\end{remark}

\begin{remark}
For $x=t_0q^h$, $h=0,1,2,\ldots$, the degree of the branching polynomials $P_{\lambda/\mu}(x; q ,t ,\mathbf{t})$ \eqref{br-pol} remains bounded by $h$ as the sum in question truncates beyond $k=h$.  In particular, for $x=t_0$ only the first (constant) term survives and the complexity of the branching coefficients reduces considerably
(cf. \cite[p. 100]{rai:bcn-symmetric}):
\begin{equation}
P_{\lambda/\mu}(t_0; q ,t ,\mathbf{t})=B_{\lambda/\mu}^0 (q,t,\mathbf{t})=
 (-1)^{|\lambda|-|\mu|} C^{(n+1)^m-\lambda^\prime,m}_{n^m-\mu^\prime,m}(t,q,\mathbf{t})  .
\end{equation}
\end{remark}

\section{Proof of the branching formula}\label{sec4}
Let
\begin{equation}
\prod (x_1,\ldots,x_m;z_1,\ldots ,z_n) :=\prod_{\substack{1\leq i\leq m\\1\leq j\leq n}} \langle x_i;z_j\rangle .
\end{equation}
Mimachi's Cauchy formula \cite[Thm. 2.1]{mim:duality} states that this kernel expands in terms of
Macdonald-Koornwinder polynomials as:
\begin{eqnarray}\label{cauchy:mim}
\lefteqn{\prod (x_1,\ldots,x_m;z_1,\ldots ,z_n)= }&& \\
&& \sum_{\lambda\subset n^m} (-1)^{mn-|\lambda |}
P_\lambda (x_1,\ldots ,x_m;q,t,\mathbf{t}) P_{m^n-\lambda^\prime}  (z_1,\ldots ,z_n;t,q,\mathbf{t}) .\nonumber
\end{eqnarray}
A similar expansion of the kernel at issue in terms of Okounkov's hyperoctahedral interpolation polynomials is given by the Cauchy formula in \cite[Thm. 6.2]{ouk:bc-type} (with shifted variables as in \cite[Thm. 3.16]{rai:bcn-symmetric}). For $n=1$, the latter Cauchy formula becomes of the form \cite[Lem. 5.1]{kom-nou-shi:kernel}:
\begin{equation}\label{cauchy:ouk}
\prod (x_1,\ldots,x_m;z)=
\sum_{0\leq r\leq m} (-1)^{m-r} E_{r}(x_1,\ldots,x_m;t,t_0) \,  \langle z;t_0\rangle_{t,m-r} .
\end{equation}
By expanding the first two factors of the trivial identity
\begin{equation*}
\prod (x_1,\ldots,x_m;z_1,\ldots ,z_n,z)=\prod (x_1,\ldots,x_m;z_1,\ldots ,z_n)\prod (x_1,\ldots,x_m;z)
\end{equation*}
by means of Mimachi's Cauchy formula \eqref{cauchy:mim} and the last factor by means of the `column-row' case of Okounkov's Cauchy formula in Eq. \eqref{cauchy:ouk}, one arrives at the equality
\begin{equation*}
\sum_{\lambda\subset (n+1)^m} (-1)^{m(n+1)-|\lambda |}
P_\lambda (x_1,\ldots ,x_m;q,t,\mathbf{t}) P_{m^{n+1}-\lambda^\prime} (z_1,\ldots ,z_n,z;t,q,\mathbf{t}) 
\end{equation*}
\begin{align*}
 = \sum_{\substack{\mu\subset n^m\\ 0\leq r\leq m}} (-1)^{m(n+1)-|\mu |-r}& \Bigl( 
P_{m^n-\mu^\prime} (z_1,\ldots ,z_n;t,q,\mathbf{t})  \langle z;t_0\rangle_{t,m-r}  \\
& \times  E_{r}(x_1,\ldots,x_m;t,t_0) P_{\mu} (x_1,\ldots ,x_m;q,t,\mathbf{t})
  \Bigr) .
\end{align*}
Upon rewriting the RHS with the aid of the Pieri formula \eqref{pieri}
\begin{align*}
 = \sum_{\substack{\mu\subset n^m\\ 0\leq r\leq m}} (-1)^{m(n+1)-|\mu |-r}& \Bigl( 
 P_{m^n-\mu^\prime} (z_1,\ldots ,z_n;t,q,\mathbf{t})  \langle z;t_0\rangle_{t,m-r}
\\
& \times    \sum_{\substack{\lambda\subset (n+1)^m \\ \lambda\sim_{r} \mu }}C^{\lambda,m}_{\mu , r}(q,t,\mathbf{t}) P_{\lambda} (x_1,\ldots ,x_m;q,t,\mathbf{t}) 
    \Bigr)  ,
\end{align*}
and reordering the sums
\begin{align*}
 =  & \sum_{\lambda\subset (n+1)^m }  (-1)^{m(n+1)-|\lambda |}  \Bigl( P_{\lambda} (x_1,\ldots ,x_m;q,t,\mathbf{t})  \\
& \times \sum_{\substack{\mu\subset n^m,  \, 0\leq r\leq m\\ \mu\sim_{r} \lambda}} 
 (-1)^{r+|\lambda|-|\mu|} C^{\lambda,m}_{\mu , r}(q,t,\mathbf{t})   
P_{m^n-\mu^\prime} (z_1,\ldots ,z_n;t,q,\mathbf{t})  \langle z;t_0\rangle_{t,m-r}    \Bigr)  ,
\end{align*}
one deduces by comparing with the LHS that for any $\lambda\subset (n+1)^m$:
\begin{align*}
&P_{m^{n+1}-\lambda^\prime} (z_1,\ldots ,z_n,z;t,q,\mathbf{t}) 
 =    \\
&  \sum_{\substack{\mu\subset n^m, \, 0\leq r\leq m\\ \mu\sim_{r} \lambda }} 
 (-1)^{r+|\lambda|-|\mu|} C^{\lambda,m}_{\mu , r}(q,t,\mathbf{t})   
P_{m^n-\mu^\prime} (z_1,\ldots ,z_n;t,q,\mathbf{t})  \langle z;t_0\rangle_{t,m-r}      ,
\end{align*}
i.e. for any $\lambda\subset m^{n+1}$:
\begin{align}\label{beq}
&P_{\lambda} (z_1,\ldots ,z_n,z;q,t,\mathbf{t}) 
 =    \\
&  \sum_{\substack{\mu\subset m^n, \, 0\leq r\leq m \\ n^m- \mu^\prime \sim_{r} (n+1)^m-\lambda^\prime }} 
 (-1)^{m-r+|\lambda|-|\mu|} C^{(n+1)^m-\lambda^\prime,m}_{n^m-\mu^\prime ,r}(t,q,\mathbf{t})   
P_{\mu} (z_1,\ldots ,z_n;q,t,\mathbf{t})  \langle z;t_0\rangle_{q,m-r}      .\nonumber
\end{align}
To finish the proof we invoke the following lemma.
\begin{lemma}
Let  $\lambda\subset m^{n+1}$ and $\mu\subset m^n$. Then
$n^m- \mu^\prime \sim_{r} (n+1)^m-\lambda^\prime$ iff
$\mu\preceq \lambda$ and
 $m-d\leq r\leq m$ with $d=\{ 1\leq j\leq m\mid \lambda_j^\prime=\mu_j^\prime +1\}$. 
\end{lemma}
\begin{proof}
The statement of the lemma is immediate upon combining the following two properties: (i) $n^m- \mu^\prime \sim_{m} (n+1)^m-\lambda^\prime$ iff $\mu\preceq \lambda$ and
(ii) $n^m- \mu^\prime \sim_{r} (n+1)^m-\lambda^\prime$  iff $n^m- \mu^\prime \sim_{m} (n+1)^m-\lambda^\prime$ and
$r\geq m-d$. 

Firstly, $n^m- \mu^\prime \sim_{m} (n+1)^m-\lambda^\prime$ $\Leftrightarrow$ $\exists\nu\subset n^m$ such that
$(n^m- \mu^\prime)/\nu$ and $((n+1)^m-\lambda^\prime)/\nu$ are vertical strips 
$\Leftrightarrow$ $\exists\kappa\subset m^n$ such that
$(m^n- \mu)/\kappa$ and $(m^{n+1}-\lambda)/\kappa$ are horizontal strips 
$\Leftrightarrow$ $\exists\kappa\subset m^n$ such that
$(m^n- \kappa)/\mu$ and $\lambda/(m^n-\kappa)$ are horizontal strips 
$\Leftrightarrow$ $\mu\preceq\lambda$, which proves part (i).

Secondly---assuming that $n^m- \mu^\prime \sim_{m} (n+1)^m-\lambda^\prime$ and picking  $\nu\subset n^m$ such that $(n^m- \mu^\prime )/\nu $ and $( (n+1)^m-\lambda^\prime)/\nu$ are vertical strips with
$| (n^m- \mu^\prime)/\nu |+| ( (n+1)^m-\lambda^\prime)/\nu | $ minimal---one has that
$| (n^m- \mu^\prime)/\nu |+| ( (n+1)^m-\lambda^\prime)/\nu |\leq r$ $\Leftrightarrow$
$|\{ 1\leq j\leq m \mid \lambda_j^\prime\neq \mu_j^\prime + 1\}|\leq r$ $\Leftrightarrow$
 $m-d\leq r$, which completes the proof of part (ii).
\end{proof}
The upshot is that Eq. \eqref{beq} can be rewritten as
\begin{align*}
&P_{\lambda} (z_1,\ldots ,z_n,z;q,t,\mathbf{t}) 
 =    \\
&  \sum_{\substack{\mu\subset m^n,\, \mu\preceq\lambda \\ m-d\leq r\leq m }} 
 (-1)^{m-r+|\lambda|-|\mu|} C^{(n+1)^m-\lambda^\prime,m}_{n^m-\mu^\prime ,r}(t,q,\mathbf{t})   
P_{\mu} (z_1,\ldots ,z_n;q,t,\mathbf{t})  \langle z;t_0\rangle_{q,m-r}      \nonumber
\end{align*}
and Theorem \ref{branching:thm}  follows (where our choice of picking $m$ equal to $\lambda_1$ corresponds to the minimal value of $m$ such that $\lambda\subset m^{n+1}$, cf. Remark \ref{rem-m} at the end of the previous section).

\section*{Acknowledgments} We are grateful to Alexei Borodin and Ivan Corwin for encouraging us to 
pursue an explicit branching formula for the Macdonald-Koornwinder polynomials.

\bibliographystyle{amsplain}

\end{document}